\documentclass[letterpaper, 10 pt, conference]{ieeeconf}  

\IEEEoverridecommandlockouts                              
\usepackage{graphicx}
\usepackage{amsmath}
\usepackage{amsfonts}
\usepackage{amssymb}
\usepackage{enumerate}
\usepackage{subcaption}

\newtheorem{definition}{Definition}
\newtheorem{lemma}{Lemma}

\newtheorem{theorem}{Theorem}
\newtheorem{proposition}{Proposition}
\newtheorem{remark}{Remark}

\newcommand\RE{\mathbb{R}}
\newcommand\PA{\mathbb{P}_\mathrm{a}}
\newcommand\PS{\mathbb{P}_\mathrm{s}}
\newcommand\rd{\mathrm{d}}
\newcommand\re{\mathrm{e}}
\newcommand\MA{\mathcal{M}}
\newcommand\SL{\mathcal{D}}
\newcommand\so{\mathfrak{so}}

\DeclareMathOperator{\vex}{vex}
\DeclareMathOperator{\tr}{tr}
\DeclareMathOperator{\diag}{diag}
\DeclareMathOperator{\sign}{sign}

\title{\LARGE \bf
 Sliding Motions on $SO(3)$, \\ Sliding Subgroups
}

\author{Gian C. G\'omez Cort\'es, Fernando Casta\~nos and Jorge D\'avila%
 \thanks{G. C. G\'omez Cort\'es and Fernando Casta\~nos are with the Automatic Control Department, 
  Cinvestav-IPN, Mexico City, Mexico.
  {\tt\small ggomez@ctrl.cinvestav.mx}, {\tt\small castanos@ieee.org}
 }
 \hspace*{ 0.5 in}
 \thanks{Jorge D\'avila is with Section of Graduate Studies and Research IPN, ESIME-UPT, Mexico City, Mexico.
  {\tt\small jadavila@ipn.mx}
 }
}

\begin{document}

\maketitle
\thispagestyle{empty}
\pagestyle{empty}

\begin{abstract}
 We propose a sliding surface for systems on the Lie group $SO(3)\times\RE^3$.
 The sliding surface is shown to be a Lie subgroup. The reduced-order dynamics
 along the sliding subgroup have an almost globally asymptotically stable equilibrium.
 The sliding surface is used to design a sliding-mode controller for the attitude
 control of rigid bodies. The closed-loop system is robust against matched disturbances
 and does not exhibit the undesired unwinding phenomenon.
\end{abstract}

\section{INTRODUCTION}

The robustness of sliding-mode controllers is a well-known feature that has
been extensively documented in the literature. For systems with Euclidean state
spaces, the main design principle is well understood: constrain the system trajectories to a,
usually, linear sliding surface. Along the sliding surface, the equations of motion
are of reduced order and, more importantly, completely independent of matched
disturbances. By the linear nature of the sliding surface, the state space
of the reduced-order system inherits the Euclidean nature of the full-order
system.  

However, there are systems which do not evolve naturally on Euclidean spaces.
The state space may not be linear, but it may possess other 
important algebraic structures. Such is the case, e.g., of state spaces
which are Lie groups, that is, state spaces which are smooth and satisfy
the group axioms for a given smooth product. An interesting example 
is $SO(3)$, which is the natural configuration space for the attitude 
of rigid bodies. In such a case it seems more intuitive to design a sliding
surface which retains the Lie group structure of the full-order system,
in other words, a sliding surface which is a Lie subgroup (a \emph{sliding
subgroup}), rather than a linear subspace. In this paper,
we propose such a sliding surface for the control of the attitude of
rigid bodies.  

The rigid body idealization has a considerable amount of applications:
robots, spacecraft, aircraft, satellites and underwater vehicles can
be modeled as rigid bodies~\cite{schaub,murray}. Problems such as attitude
determination, control and estimation can be solved using rigid body
models~\cite{markley,schaub}.

The motion of a rigid body is represented via Euler equations of motion
and kinematic equations which depend on the attitude parametrization~\cite{murray,shuster1993}.
The attitude of a rigid body can be defined using different representations such as
rotation matrices, Euler angles, Rodrigues parameters, modified Rodrigues parameters,
quaternions, axis-angle, and many others. Surveys of the attitude representations are
given in~\cite{shuster1993} and~\cite{phillips2001}, where the relations between attitude parametrizations
and kinematic equations are detailed.

A natural, unique, and global way to describe the rigid-body attitude is by means of rotation
matrices. These matrices form a matrix group called Special Orthogonal, which in the case of
3-D motions is denoted as $SO(3)$. Many researchers have studied the control problem
in $SO(3)$. Within the literature, we can mention~\cite{brockett1973}, where 
properties like controllability and observability of systems defined on spheres are treated.
A proportional-derivative control scheme for systems defined on $SO(3)$ and $SE(3)$ is given
in~\cite{bullo1995}, and solutions to the output regulation problem for systems defined on
matrix Lie groups are presented in~\cite{mahony2015,marconi2018}.

One peculiarity of the control problem in $SO(3)$ is the impossibility to achieve global 
stability results using continuous time-invariant feedback, due to the topology of
$SO(3)$~\cite{koditschek1989,bhat2000,chaturvedi2011}. This problem calls for the notion
of almost-global stability, defined in~\cite{koditschek1989} and which is typically needed
in systems with rotational motion. Within the works which emphasize the almost-global
stability property, we find~\cite{chaturvedi2011,mahony2015,bullo1995,sanyal2009,pong2015,bullomurray1995}.
The attitude control problem for different systems is considered and almost-global stability
is achieved. 

\subsection{Contributions}

We propose a sliding surface parametrized with rotational matrices. This automatically
ensures that the closed-loop system is globally well defined and that it does not
exhibit the so-called unwinding phenomenon. We show that the sliding surface is 
a Lie group, so it inherits the topological and algebraic properties of the system
state space. We show that the controller achieves almost global asymptotic stability,
being the global asymptotic stability impossible in the light of the non-contractiveness
of $SO(3)$.

\subsection{Paper Structure}

In Section~\ref{sec:motivation} we consider the attitude control problem about a single 
axis of rotation. We show that, if we insist on a linear surface in a Euclidean coordinate
chart, the sliding surface will necessarily break into disjoint subsets. Section~\ref{sec:prel}
contains all the preliminary material, mostly about $SO(3)$ and the dynamics of a rotating 
rigid body. The controller is proposed and the stability of the closed loop system is analyzed 
in Section~\ref{sec:main}. The performance of the controller is further assessed with simulations in
Section~\ref{sec:simulation}. Conclusions are given in the last section.

\section{MOTIVATIONAL EXAMPLE, \\ SLIDING MOTIONS ON $S\times\RE$} \label{sec:motivation}

To illustrate the potential problems of linear sliding surfaces
on systems with non-Euclidean configuration spaces we elaborate on an example
given in~\cite{bhat2000}. Although the scope of~\cite{bhat2000} is restricted to
smooth control laws, the essence of the problems remains in the non-smooth setting.

Consider a rigid body with a single fixed axis of rotation. Its phase space $\MA$
consists of all possible rotations, $S$, together with all possible angular velocities,
$\RE$, so that $\MA$ is the cylinder $S\times\RE$. Suppose further that the system
is subject to a control torque, $u$. In local coordinates $(\theta, \omega$)
for $\MA$, the equations of motion are 
\begin{subequations} \label{eq:systemSR}
\begin{align}
 \dot{\theta} &= \omega \\
 \dot{\omega} &= u \;.
\end{align}
\end{subequations}

The objective is to regulate the angular position to the value $\theta=0$.
A typical linear sliding variable could be 
\begin{equation} \label{eq:standardSV}
 \sigma(\theta,\omega) = \omega + \theta \;.
\end{equation}
It is straightforward to verify that the control law
\begin{equation} \label{eq:uS1}
 u(\theta,\omega) = -\left( |\omega|+1 \right)\sign \sigma(\theta,\omega)
\end{equation}
drives the state to the sliding surface 
\begin{displaymath}
 \SL = \left\{ (\theta,\omega) \in \RE\times\RE \mid \sigma(\theta,\omega) = 0 \right\}
\end{displaymath}
(solutions are taken in the sense of Filippov~\cite{filippov}).
Also, it is easily seen that the reduced-order dynamics along the sliding surface
are given by $\dot{\theta} = -\theta$,
for which $\theta = 0$ is globally exponentially stable. It may thus appear
that the closed-loop system~\eqref{eq:systemSR}-\eqref{eq:uS1} has a globally exponential
stable equilibrium. However, note that a given point is not represented 
uniquely by $\theta$. Namely, given a principal angle $\bar{\theta} \in [0, 2\pi)$, all $\theta$ such
that $\theta \bmod 2\pi = \bar{\theta}$ represent the same point in $S$. Moreover, some of these
$\theta$ yield different control values. Thus, the control is, in reality, multi-valued, 
a fact that has been ignored in the previous analysis and which invalidates the hasty
conclusion about the global asymptotic stability of $(0,0)$.

A direct practical consequence of the non-unique representation is the 
\emph{unwinding} phenomenon. To appreciate it, consider an initial condition $(4\pi,0) \not\in \SL$.
Although this initial condition represents precisely the desired equilibrium
$(0,0) \in \SL$, the control law~\eqref{eq:uS1} will drive the system state
from $(4\pi,0)$ to $(0,0)$ by making two absolutely unnecessary rotations.

We can stop the unwinding phenomenon if we replace~\eqref{eq:standardSV} with
\begin{equation} \label{eq:linearSV}
 \sigma(\theta,\omega) = \omega - \pi + (\theta - \pi)\bmod2\pi \;,
\end{equation}
as this makes the control single-valued. Note that the sliding variable is still linear
for $\theta \in [0, 2\pi]\setminus\left\{ \pi \right\}$ but, because of the topology of the cylinder, it
must necessarily contain a discontinuity at some point in order to retain its linear character. In this case,
we have chosen to place such discontinuity at $\theta = \pi$.

Figure~\ref{fig:linear} shows the phase plane of~\eqref{eq:systemSR},~\eqref{eq:uS1},~\eqref{eq:linearSV}.
For this simple two-dimensional example, it suffices to look at the phase plane to see that the system enjoys
almost global asymptotic stability (see Definition~\ref{def:almost} on p.~\pageref{def:almost}), which is the
most we can expect from a mechanical system with a compact 
configuration space. However, note that $\SL$ is not connected, and that the vector field has discontinuities 
outside $\SL$, along the dashed line joining the points $(\pi,-\pi)$ and $(\pi,\pi)$. Also, note that there
is an unstable equilibrium at $(\pi,0)$ and that its presence rules out the global attractivity to $(0,0)$.
A rigorous analysis of the global stability properties of this system requires the use of mathematical tools
not commonly found in the sliding-mode literature. Indeed, with few exceptions such as~\cite{miranda},
$\sigma$ is always required to be continuous.

\begin{figure}
 \begin{subfigure}[t]{0.23\textwidth}
  \centering
  \includegraphics[width=\linewidth]{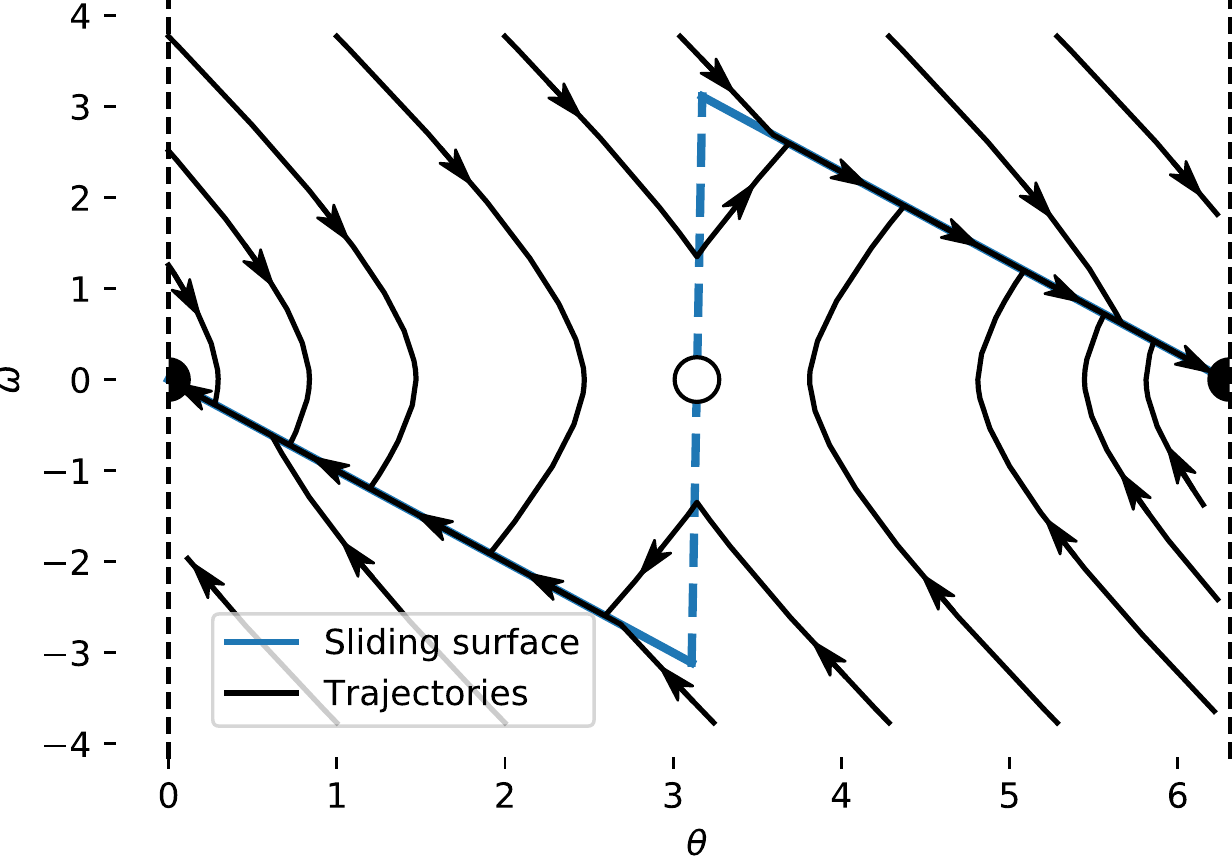}
  \caption{Sliding variable~\eqref{eq:linearSV}. The sliding surface
  is disconnected and the vector field is discontinuous at some 
  point outside it.}
  \label{fig:linear}
 \end{subfigure}
 \quad
 \begin{subfigure}[t]{0.23\textwidth}
  \centering
  \includegraphics[width=\linewidth]{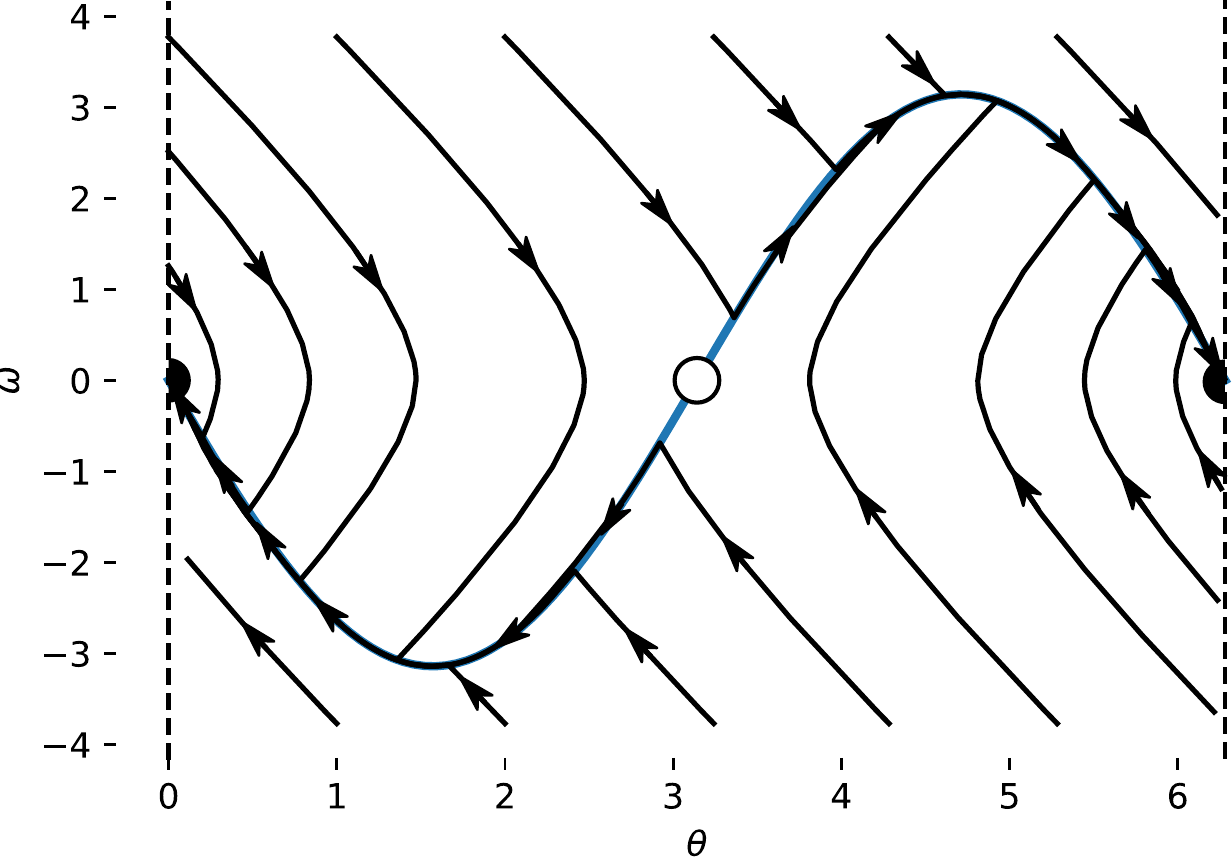}
  \caption{Sliding variable~\eqref{eq:contSV}. The sliding surface
  is connected and smooth.}
  \label{fig:continuous}
 \end{subfigure}
 \caption{Phase planes of~\eqref{eq:systemSR},~\eqref{eq:uS1}. We identify pairs of 
  points $(0,\omega)$ and $(2\pi,\omega)$, so that the phase planes lie on $S\times\RE$.
  There are a stable and an unstable equilibrium at $(0,0)$ and $(0,\pi)$, respectively.}
 \label{fig:S1}
\end{figure}

A reasonable alternative is to replace~\eqref{eq:standardSV} by
\begin{equation} \label{eq:contSV}
 \sigma(\theta,\omega) = \omega + \pi \sin\theta \;.
\end{equation}
The periodicity in $\theta$ ensures that~\eqref{eq:uS1} is single-valued without destroying 
the continuity of the sliding variable. The phase plane is shown in Figure~\ref{fig:continuous}.
Again, $(0,0)$ is almost globally asymptotically stable, but without the technical difficulties mentioned
above.
	
In Section~\ref{sec:main} we propose a smooth sliding variable for systems whose configuration
space is $SO(3)$. The sliding variable is such that, in the same spirit of~\eqref{eq:contSV},
yields a single-valued controller that exhibits no unwinding. The stability properties of the system
will be established rigorously (see Remark~\ref{rem:SR}).


\section{PRELIMINARIES, \\ SYSTEMS ON $SO(3)\times\RE^3$} \label{sec:prel}

In this section, we derive the error equations for the attitude of a rigid body. We also
recall some stability notions.

\subsection{$SO(3)$}

Since the set of possible attitudes of a rigid body is not a Euclidean space, several
parametrizations are used in the literature. Common parametrizations include Euler angles, quaternions, and 
rotational matrices. Parametrizations using Euler angles contain singular points, that is,
points at which the Jacobian of the coordinate chart loses rank and angular velocities
are no longer well defined. Also, Euler angles are non-unique, which leads to the unwinding
phenomenon described in the previous section. Parametrizations using quaternions are global,
by which we mean that there are no singular points. However, quaternion representations are
still non-unique (as every attitude is represented by two quaternions) and the unwinding problem
can still potentially manifest. 

The set of rotation matrices forms the Special Orthogonal group
\begin{displaymath}
 SO(3) = \left\{ R \in\mathbb{R}^{3\times 3} \mid R R^\top = I,\; \det(R) = 1 \right\}
\end{displaymath}
with $I$ the identity matrix of order 3. Rotational matrices parametrize the set of attitudes,
inasmuch as every attitude is associated with a single rotational matrix and \emph{vice versa}. 
Parametrizations using rotational matrices have the compelling advantage of being both global
and unique. The interested reader is referred to~\cite{chaturvedi2011}, where a more thorough
analysis of these different parametrizations can be found.
 
Besides being a smooth manifold, $SO(3)$ has the structure of a Lie group with the standard matrix
product. The tangent space at the identity of $SO(3)$, denoted $T_ISO(3)$, can be associated with 
the space of skew-symmetric matrices $X \in \RE^{3\times 3}$, $X + X^\top = 0$. Endowing this 
space with the bracket $[A,B]=AB-BA$ gives rise to the Lie algebra of $SO(3)$, denoted
by $\so(3)$ (see~\cite[Appx. A]{murray} for details and formal definitions). 

Let $\times$ be the standard cross product in $\RE^3$. We will make use of the isomorphism 
$\omega \mapsto \omega^\times$ between $(\RE^3,\times)$ and $(\so(3),[\cdot,\cdot])$, defined
as
\begin{displaymath}
 \omega^\times = 
  \begin{pmatrix}
           0 & -\omega_3 &  \omega_2 \\
    \omega_3 &         0 & -\omega_1 \\
   -\omega_2 &  \omega_1 & 0
  \end{pmatrix} \;.
\end{displaymath}
For every $(R,\omega)$, we have~\cite{mahony2008}
\begin{equation} \label{eq:relVel}
 (R \omega)^\times = R \omega^\times R^\top \;.
\end{equation}
We will also make use of the inverse map of $(\cdot)^\times$, denoted by $\vex : \so(3) \to \RE^3$.

The operator $\PA : \RE^{3\times 3} \to \so(3)$ retains the skew symmetric part of a matrix,
\begin{displaymath}
 \PA(A) = \frac{1}{2}\left( A - A^\top \right) \;,
\end{displaymath}
while $\PS(A) : \RE^{3\times 3} \to \RE^{3\times 3}$,
\begin{displaymath}
 \PS(A) = \frac{1}{2}\left( A + A^\top \right) \;,
\end{displaymath}
retains the symmetric part. 

The trace will serve as a matrix inner product on $\RE^{3\times 3}$,
\begin{displaymath}
 \langle A, B \rangle = \tr(A^\top B) \;.
\end{displaymath}
If $v,w \in \RE^3$ and $A = A^\top$, we know that~\cite{mahony2008}
\begin{equation} \label{eq:inner}
 \langle v, w \rangle = \frac{1}{2} \langle v^\times, w^\times \rangle 
\end{equation}
and
\begin{equation} \label{eq:symSkew}
 \langle A, v^\times \rangle = 0 \;.
\end{equation}

Recall that, by Euler's rotation theorem, any attitude $R$ can be attained by
a single rotation $\theta \in [0,2\pi)$ about some unit vector $\eta \in \RE^3$.
Rodrigues formula establishes that
\begin{equation} \label{eq:rodrigues}
 R = I + \eta^\times \sin(\theta) + (\eta^\times)^2(1-\cos(\theta)) \;.
\end{equation}
It follows from Rodrigues formula that~\cite[Ch. 2]{murray}
\begin{equation} \label{eq:trace}
 \tr(R) = 1 + 2\cos(\theta) \;.
\end{equation}
 
\subsection{Attitude Dynamics}

The state space $\MA$ of the attitude dynamics consists of the set of possible attitudes,
$SO(3)$, together with the set of possible angular velocities, $\RE^3$, so that 
$\MA = SO(3)\times\RE^3$.

Suppose we are able to provide torques along the principal axes of a rigid body 
with an inertia matrix $J \in \RE^{3\times 3}$, $J = J^\top > 0$.
Then, the rotational motion of the rigid body is determined by a kinematic equation
\begin{displaymath}
 \dot{R} = R\omega^\times 
\end{displaymath}
(for conciseness we omit all time arguments), and Euler's equation,
\begin{displaymath}
 J\dot{\omega} = (J\omega)\times \omega + u + d \;,
\end{displaymath}
where $u,d \in \RE^3$ are the control and the disturbance torques, respectively~\cite[Ch. 2]{murray}.

Let $R_\rd$ and $\omega_\rd$ be (possibly time varying) desired attitude and angular velocity
satisfying the kinematic equation
\begin{equation} \label{eq:desired}
 \dot{R}_\rd = R_\rd \omega_\rd^\times \;. 
\end{equation}
Now, define the attitude and angular velocity errors $R_\re = R_\rd^\top R$ and
$\omega_\re = \omega - R_\re^\top\omega_\rd$, respectively, and note that
\begin{displaymath}
 (R_\re,\omega_\re) = (I, 0) \quad \text{if, and only if,} \quad (R,\omega) = (R_\rd, \omega_\rd) \;.
\end{displaymath}
The error kinematics are then
\begin{align*}
 \dot{R}_\re &= (R_\rd \omega_\rd^\times)^\top R + R_\rd^\top R \omega^\times \\
             &= -(\omega_\rd)^\times R_\re + R_\re\omega^\times \\
             &= (R_\re(\omega_\re-\omega))^\times R_\re + R_\re\omega^\times \;.
\end{align*}
Using~\eqref{eq:relVel} we obtain
\begin{subequations} \label{eq:error}
\begin{equation} 
 \dot{R}_\re = R_\re \omega_\re^\times \;.
\end{equation}
The dynamic equation is
\begin{align*}
 J\dot{\omega}_\re &= J(\dot{\omega}-\dot{R}_\re^\top\omega_\rd - R_\re^\top\dot{\omega}_\rd) \\
                   &= (J\omega)\times\omega + J\omega_\re^\times R_\re^\top \omega_\rd - JR_\re^\top\dot{\omega}_\rd + u + d \\
                   &= (J\omega)\times\omega + JR_\re^\top\left((R_\re\omega_\re)^\times\omega_\rd - \dot{\omega}_\rd \right)
                    + u + d\;,
\end{align*}
where we have used~\eqref{eq:relVel} to obtain the last equality. Set
\begin{displaymath}
 u = -JR_\re^\top\left((R_\re\omega_\re)^\times\omega_\rd - \dot{\omega}_\rd \right) + v
\end{displaymath}
so that
\begin{equation}
 J\dot{\omega}_\re = (J\omega)\times\omega + v + d \;,
\end{equation}
\end{subequations}

\subsection{Stability}

It follows from the compacity of $SO(3)$ and the arguments presented in~\cite{bhat2000} that
global asymptotic stability of an equilibrium of~\eqref{eq:error} is impossible,
so we must settle for \emph{almost} global asymptotic stability, a property that we recall now.

\begin{definition} \label{def:almost}
A dynamical system is almost globally asymptotically stable if all trajectories starting in 
some open dense subset of the state space tend asymptotically to a specified stable equilibrium state.
\end{definition}

Almost global stability will be proved using a refinement of the LaSalle
invariance principle for systems defined on manifolds. 

\begin{theorem}[LaSalle principle~\cite{bullo}] \label{thm:laSalle}
For a smooth vector field $f$ on $\MA$, let $\mathcal{A} \subset \MA$ be compact and
positively invariant for $f$. Let a smooth function $V:\MA \to \mathbb{R}$
satisfy $\dot{V}(x)\leq 0$ for all $x \in \mathcal{A}$ and let $\mathcal{B}$ be the 
largest positively invariant set for $f$ contained in
$\{ x \in \mathcal{A} \mid \dot{V}(x) = 0 \}$. Then, the following
statements hold:
\begin{itemize}
 \item Each integral curve of $f$ with initial condition in $\mathcal{A}$ 
  approaches $\mathcal{B}$ as $t \to +\infty$.
 \item If $\mathcal{B}$ consists of a finite number of isolated points, then
  each integral curve of $f$ with initial condition in $\mathcal{A}$ converge to a point of
  $\mathcal{B}$ as $t \to -\infty$.
\end{itemize}
\end{theorem}

%


\section{MAIN RESULTS, \\ SLIDING-MODE CONTROL ON $SO(3)\times\RE^3$} \label{sec:main}

Allow us first to endow $\MA = SO(3)\times \RE^3$ with the product
$\MA \times \MA \to \MA$ defined by
\begin{subequations} \label{eq:product}
\begin{equation}
 (R_1,\omega_1)\cdot(R_2,\omega_2) = (R_3,\omega_3) \;,
\end{equation}
where
\begin{align}
             R_3 &= R_1 R_2 \\
 \omega_3^\times &= \PA\left(R_1 \omega_2^\times + R_2^\top \omega_1^\times - \frac{1}{2}[R_1,R_2^\top] \right) \;.
\end{align}
\end{subequations}

\begin{lemma} \label{lem:group}
 $\MA$ is a Lie group with the product~\eqref{eq:product}.
\end{lemma}

\begin{proof}
 $\MA$ is a smooth manifold and~\eqref{eq:product} is clearly smooth,
 so we will only verify the group axioms.
 
 The point $(I,0)$ is the identity:
 \begin{displaymath}
  (I,0)\cdot(R,\omega) = (R,\vex(\PA(\omega^\times I+0))) = (R,\omega) \;. 
 \end{displaymath}
 
 The multiplicative inverse of $(R,\omega)$ is $(R^\top,-\omega)$:
 \begin{multline*}
  (R^\top,-\omega)\cdot(R,\omega) = \left(I, \vex(\PA(R^\top\omega^\times - R^\top\omega^\times + 0))\right) \\
    = (I,0) \;.
 \end{multline*} 
\end{proof}

\subsection{The Proposed Sliding Subgroup}

The sliding surface is introduced in the following proposition.
\begin{proposition}
 Consider the sliding variable $\sigma : \MA \to \RE^3$ defined by
 \begin{equation} \label{eq:sigma}
  \sigma(R_\re,\omega_\re) = \omega_\re + \vex\left( \PA(R_\re) \right) \;.
 \end{equation}
 The sliding surface
 \begin{equation} \label{eq:D}
  \SL = \left\{ (R_\re,\omega_\re) \in \MA \mid \sigma(R_\re,\omega_\re) = 0 \right\}
 \end{equation}
 is a Lie subgroup of $\MA$.
\end{proposition}

\begin{proof}
 The smoothness of $\SL$ is a simple consequence of the smoothness of $\vex$ and $\PA$.
 We now verify the group axioms.

 The identity is in $\SL$: $\sigma(I,0) = 0 + \vex\left( \PA(I) \right) = 0$.
 
 $\SL$ is closed under multiplication: The product~\eqref{eq:product} gives 
 \begin{multline*}
  \sigma\left((R_1,\omega_1)\cdot(R_2,\omega_2)\right)^\times = \\
   \PA\left(R_1 \omega_2^\times + R_2^\top \omega_1^\times - \frac{1}{2}[R_1,R_2^\top] \right) + \PA(R_1R_2) \;.
 \end{multline*}
 Now, let $(R_1,\omega_1)$ and $(R_2,\omega_2)$ be both in $\SL$. We have
 \begin{multline*}
  \sigma\left((R_1,\omega_1)\cdot(R_2,\omega_2)\right)^\times = \\
   \PA\left( R_1R_2 - \frac{1}{2}[R_1,R_2^\top] - R_1\PA(R_2) - R_2^\top \PA(R_1)\right) \;.
 \end{multline*}  
 Since
 \begin{multline*}
   R_1R_2 - \frac{1}{2}[R_1,R_2^\top] - R_1\PA(R_2) - R_2^\top \PA(R_1) =  \\
    R_1R_2 + R_2^\top R_1^\top
 \end{multline*}
 and the right-hand side is symmetric, we have 
 \begin{displaymath}
  \sigma\left((R_1,\omega_1)\cdot(R_2,\omega_2)\right) = 0 \;,
 \end{displaymath}
 which shows that $(R_1,\omega_1)\cdot(R_2,\omega_2)$ is also in $\SL$.

 $\SL$ is closed under inversion: Suppose that $(R,\omega) \in \SL$. We will verify that 
 $(R^\top,-\omega) \in \SL$. Direct computation gives
 \begin{displaymath}
  \sigma(R^\top,-\omega)^\times = -\omega^\times + \PA(R^\top) = -\omega^\times - \PA(R) = 0 \;.
 \end{displaymath} 
\end{proof}

Note that the sliding surface inherits both the topological and algebraic properties of
the phase space. This stands in contrast with the standard linear approach.

\subsection{Stability of the Reduced-Order System}

The dynamics along $\SL$ are obtained by simply enforcing the constraint
$\sigma(R_\re,\omega_\re) = 0$ on~\eqref{eq:error}. These are
\begin{equation} \label{eq:reduced}
 \dot{R}_\re = -R_\re\PA(R_\re) \;.
\end{equation}

\begin{theorem} \label{thm:almostStab}
 The identity $R_\re = I$ is an almost globally stable equilibrium of the 
 reduced-order dynamics~\eqref{eq:reduced}.
\end{theorem}

\begin{proof}
 We borrow the linear Lyapunov-function candidate~\cite{mahony2008}
 \begin{displaymath}
  V_R(R_\re) = \frac{1}{2}\tr\left( I - R_\re \right) \;.
 \end{displaymath}
 Note that $V_R(R_\re) \ge 0$ by the facts that $\tr(I) = 3$ and that $\tr(R_\re) \in [-1,3]$
 (cf.~\eqref{eq:trace}),  and that $V_R(R_\re) = 0$ if, and only if, $R_\re = I$. Its time
 derivative is
 \begin{align*}
  \dot{V}_R(R_\re) &= \frac{1}{2} \tr\left( R_\re\PA(R_\re) \right) \\
                   &= \frac{1}{2} \langle R_\re^\top, \PA(R_\re) \rangle \\
                   &= \frac{1}{2} \langle -\PA(R_\re)+\PS(R_\re), \PA(R_\re) \rangle \\
                   &= -\frac{1}{2} \langle \PA(R_\re), \PA(R_\re) \rangle \;,
 \end{align*}
 where we have used~\eqref{eq:symSkew} to derive the last equation. Using~\eqref{eq:inner}
 we obtain
 \begin{align*}
  \dot{V}_R(R_\re) &= -\langle \vex(\PA(R_\re)), \vex(\PA(R_\re)) \rangle \\
                   &= -\|\vex(\PA(R_\re))\|^2 \le 0 \;.
 \end{align*}
 This proves the stability of $R_\re = I$. 

 Almost global convergence will be established using Theorem~\ref{thm:laSalle}.
 Allow us to first compute the set 
 \begin{displaymath}
  \mathcal{E} = \left\{ R_\re \in SO(3) \mid \dot{V}_1(R_\re) = 0 \right\} \;.
 \end{displaymath}
 It follows from~\eqref{eq:rodrigues} that
 \begin{displaymath}
  \PA(R_\re) = \sin(\theta)\eta_\re^\times \;.
 \end{displaymath}
 Since $\eta_\re \neq 0$, we have $\PA(R_\re) = 0$ only if $\theta \in \left\{ 0, \pi \right\}$.
 From~\eqref{eq:rodrigues} we reach the conclusion
 \begin{displaymath}
  \mathcal{E} = \left\{ I \right\}\cup\left\{ -I + 2\eta\eta^\top \mid \eta \in \RE^3,\;\|\eta\| = 1 \right\} \;.
 \end{displaymath}
 Note that $\mathcal{E}$ is comprised of equilibria, so $\mathcal{E}$ is in fact the largest invariant set
 contained in $\mathcal{E}$, so all trajectories converge to $\mathcal{E}$.

 Finally, note that for 
 \begin{displaymath}
  R_\re \in \left\{ -I + 2\eta\eta^\top \mid \eta \in \RE^3,\; \|\eta\| = 1 \right\}
 \end{displaymath}
 we have $V_R(R_\re) = \tr\left(I-\eta\eta^\top \right) = 3 - \tr(\eta^\top\eta) = 2$,
 which is the maximum value that $V_R$ can attain. This implies that all the equilibria
 $R_\re \in \mathcal{E}$ are unstable except for $R_\re = I$. Thus, the only attractor is 
 the point $R_\re = I$. The fact that $\mathcal{E}$ is of measure zero allows us to
 conclude almost global convergence.
\end{proof}

\begin{remark} \label{rem:SR}
 The problem of Section~\ref{sec:motivation} and the proposed sliding surface
 can be recovered from the $SO(3)\times\RE^3$ setting by restricting the motion about a single axis.
\end{remark}

\subsection{The Reaching Law}

The condition $\sigma(R_\re,\omega_\re) = 0$ can be enforced with the following control law.

\begin{theorem} \label{thm:reaching}
 Suppose that the perturbations $d$ are bounded by a known constant $\bar{d} > 0$, that is,
 $\|d\| \le \bar{d}$. Take a constant $\delta > 0$ and set
 \begin{equation} \label{eq:uSO3}
  v(R_\re,\omega_\re,\omega) = -K(\omega_\re,\omega) \frac{\sigma(R_\re,\omega_\re)}{\|\sigma(R_\re,\omega_\re)\|}
 \end{equation}
 with $\sigma$ as in~\eqref{eq:sigma} and with the gain
 \begin{equation} \label{eq:gain}
  K(\omega_\re,\omega) \ge \|J\|_2 \cdot\|\omega\|^2 + \|\omega_e\| + \bar{d} + \delta \;.
 \end{equation}
 Then, the trajectories of~\eqref{eq:error} converge to $\SL$~\eqref{eq:D} in finite time.
\end{theorem}

\begin{proof}
 Consider the Lyapunov candidate function
 \begin{displaymath}
  V_\sigma(\sigma) = \frac{1}{2} \sigma^\top J \sigma \;.
 \end{displaymath}
 Its time derivative is, according to~\eqref{eq:sigma},
 \begin{displaymath}
  \dot{V}_\sigma(\sigma) = \sigma^\top J\left( \dot{\omega}_\re + \vex(\PA(\dot{R}_\re)) \right) \;.	
 \end{displaymath}
 The dynamics~\eqref{eq:error} give
 \begin{align*}
  \dot{V}_\sigma(\sigma) &= \sigma^\top J\left( \dot{\omega}_\re + \vex(\PA(\dot{R}_\re)) \right) \\
                         &= \sigma^\top \left( (J\omega)\times\omega + \vex(\PA(R_\re \omega_e^\times)) + d + v \right) \;.
 \end{align*}
 It is not difficult to verify that
 \begin{displaymath}
  \left\|(J\omega)\times\omega\right\| \le \|J\|_2\cdot\|\omega\|^2 
 \end{displaymath}
 and that~\eqref{eq:inner} implies that
 \begin{displaymath}
  \| \vex(\PA(R_\re \omega_e^\times)) \| = \|\omega_e\| \;.
 \end{displaymath}
 Thus, the time derivative is bounded as
 \begin{displaymath}
  \dot{V}_\sigma(\sigma) \le -\|\sigma\|\left(K(\omega_\re,\omega) - \|J\|_2\cdot\|\omega\|^2 - \|\omega_e\| - \bar{d}  \right) \;.
 \end{displaymath}
 Condition~\eqref{eq:gain} ensures that 
 \begin{displaymath}
  \dot{V}_\sigma(\sigma) \le -\delta \sqrt{ \frac{V_\sigma(\sigma)}{\lambda_{\max}(J)}} \;.
 \end{displaymath}
 The Comparison Lemma implies that $V_\sigma$, and hence $\sigma$, go to zero in 
 finite time.
\end{proof}

\section{SIMULATIONS} \label{sec:simulation}

In this section we show the unwinding phenomenon for a controller parametrized with quaternions.
We also illustrate the performance of the proposed controller using simulations.

\subsection{The Unwinding Phenomenon in Quaternion Representations}

We compare a sliding-mode controller based on quaternions with our controller to illustrate the potential unwinding problems
of the quaternion approach.

Popular attitude representations are based on quaternions. Even though global representations are obtained  
(there are no singularities), the non-uniqueness of quaternion representations lead to undesired unwinding.
In the following example we simulate the behavior of a rigid body with a sliding-mode controller 
based on quaternions~\cite{lo1995} and with the controller~\eqref{eq:uSO3}.

In terms of the unit quaternion $\textbf{q} = (q_0,\textbf{q}_v)$, $q_0\in\mathbb{R}$, $\textbf{q}_v\in\mathbb{R}^n$,
the attitude dynamics of a rigid body are given by
\begin{equation} \label{eq:quater}
\begin{aligned}
      J\dot{\omega} &= (J\omega)\times \omega + u + d \\
          \dot{q}_0 &= -\frac{1}{2}\textbf{q}_v\omega \\
 \dot{\textbf{q}}_v &= \frac{1}{2}(q_0I+\textbf{q}_v^\times)\omega  
\end{aligned} \;.
\end{equation}

Consider the sliding-mode controller with linear sliding variable
\begin{equation} \label{eq:contq}
\begin{aligned}
 \sigma_q(\textbf{q}_v,\omega) &= \textbf{q}_v + \omega \\
        u(\textbf{q}_v,\omega) &= -k_q \frac{\sigma_q(\textbf{q}_v,\omega)}{\|\sigma_q(\textbf{q}_v,\omega)\|}
\end{aligned} \;.
\end{equation}
As before, the gain $k_q$ is chosen large enough to enforce the sliding motion (details are not shown).

For concreteness, we take $J=\diag(3,4,5)$ as the inertia matrix. The disturbances
are given by
\begin{displaymath}
 d(t) =  
 \begin{pmatrix}
  \sin(5\pi t) & \cos(7\pi t) & \sin(9\pi t)
 \end{pmatrix}^\top \;.
\end{displaymath}
The gain~\eqref{eq:gain} is set as
\begin{equation} \label{eq:gainEx}
 K(\omega_\re,\omega) \ = 7\cdot\|\omega\|^2 + 2\cdot\|\omega_e\| + 1.8 \;,
\end{equation}
while we set $k_q=5$.

We consider the regulation problem $R_d = I$, $\omega_\rd = 0$. To illustrate the unwinding phenomenon, suppose that the
system starts at the desired attitude, $R(0) = I$. In terms of quaternions, the attitude is represented
either as 
\begin{displaymath}
 \textbf{q}_1 = 
 \begin{pmatrix}
  -1 & 0 & 0 & 0
 \end{pmatrix}^\top
 \quad \text{or} \quad
 \textbf{q}_2 = 
 \begin{pmatrix}
  1 & 0 & 0 & 0
 \end{pmatrix}^\top \;.
\end{displaymath}
Although they both represent the same attitude, $\textbf{q}_1$ is unstable while $\textbf{q}_2$ is stable. 
Allow us to initiate the system on the unstable equilibrium $\textbf{q}(0) = \textbf{q}_1$.

Figure~\ref{fig:unw} shows the trace of $R_\re$ (a measure of how close $R_\re$ is to the identity) and
$\|\omega\|$ for systems~\eqref{eq:error} and~\eqref{eq:quater}, under the action of the
controllers~\eqref{eq:uSO3} and~\eqref{eq:contq}, respectively. It can be seen that the 
control law~\eqref{eq:uSO3} keeps the attitude at the desired location, i.e., where 
$\tr(R_e)=3$ and $\|\omega\|=0$. On the other hand, even though system~\eqref{eq:quater}
starts at the desired attitude, the control~\eqref{eq:contq} drives the state away from 
the desired attitude and converges again after a full turn. 

\begin{figure}
\centering
 \includegraphics[width=0.85\columnwidth]{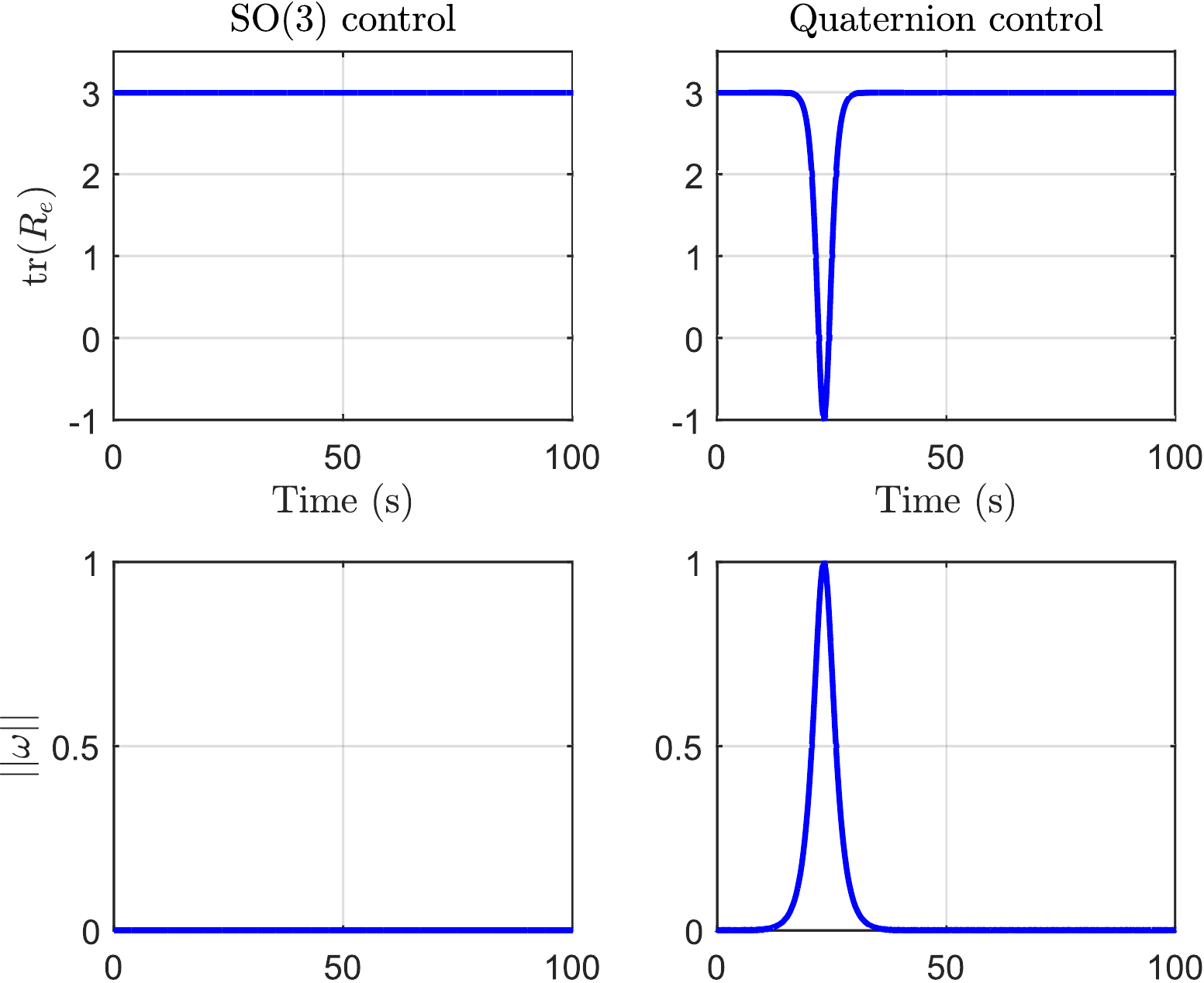}
 \caption{Comparison between sliding-mode attitude controllers. The controller based on
  quaternion representations exhibits the unwinding phenomenon, the one based on
  rotation matrices does not.}
\label{fig:unw}
\end{figure}

\subsection{Controller Performance for a Tracking Problem}

Now we consider a time-varying reference satisfying~\eqref{eq:desired} with 
the desired angular velocity shown in Figure~\ref{fig:ang_vel_track}. The figure
also shows the actual angular velocity for the system controlled
with~\eqref{eq:uSO3}. It can be seen that $\omega_e\to 0$. Figure~\ref{fig:u} shows 
the control torques. The trace of the attitude error matrix is shown in Figure~\ref{fig:trace}.
It shows that $\tr(R_e)\to 3$ which necessarily means that $R \to R_d$. 

\begin{figure}
\centering
\includegraphics[width=0.85\columnwidth]{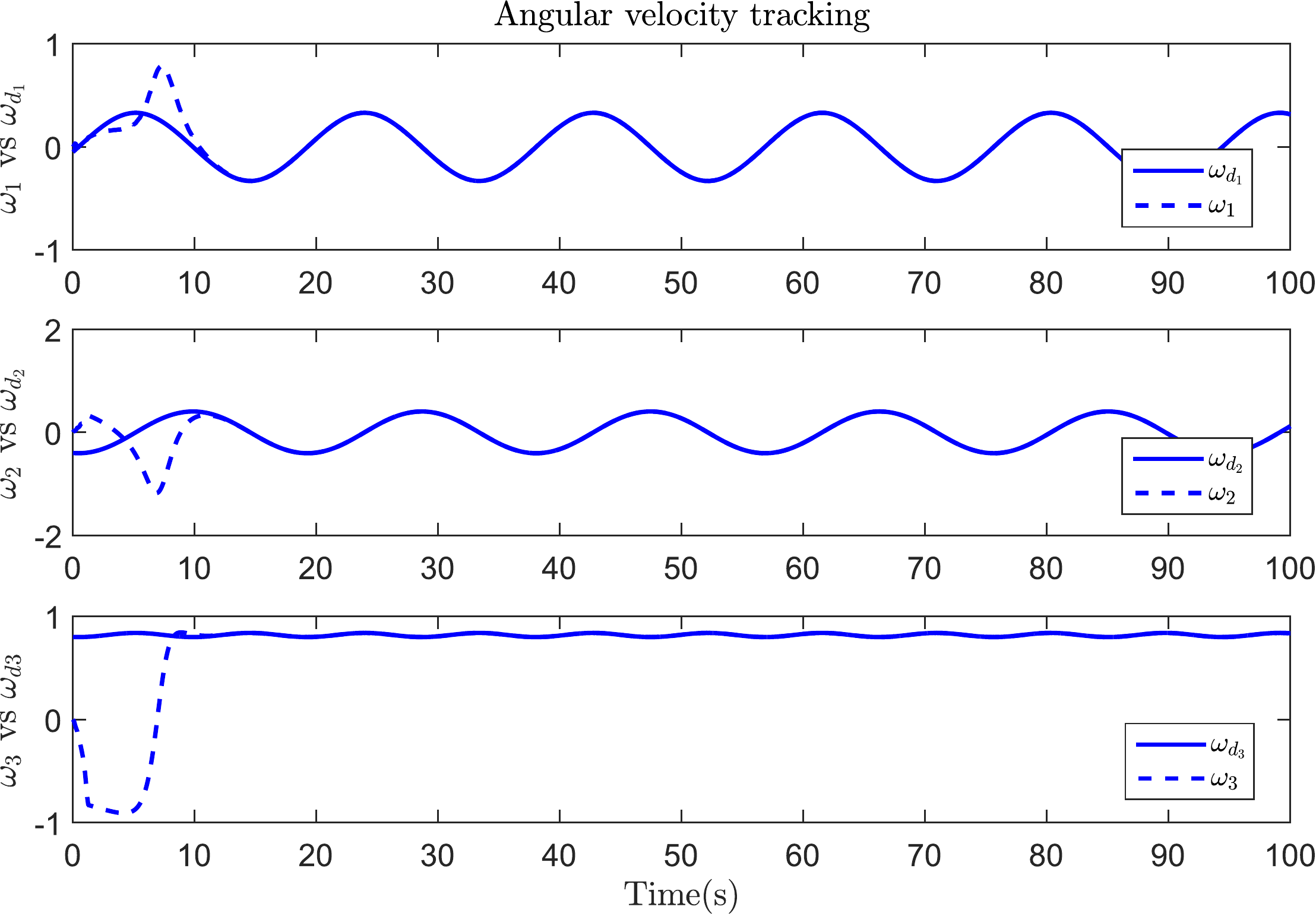}
\caption{Desired and actual angular velocities. Sliding-mode controller~\eqref{eq:uSO3} in $SO(3)$.}
\label{fig:ang_vel_track}
\end{figure}

\begin{figure}
\centering
\includegraphics[width=0.85\columnwidth]{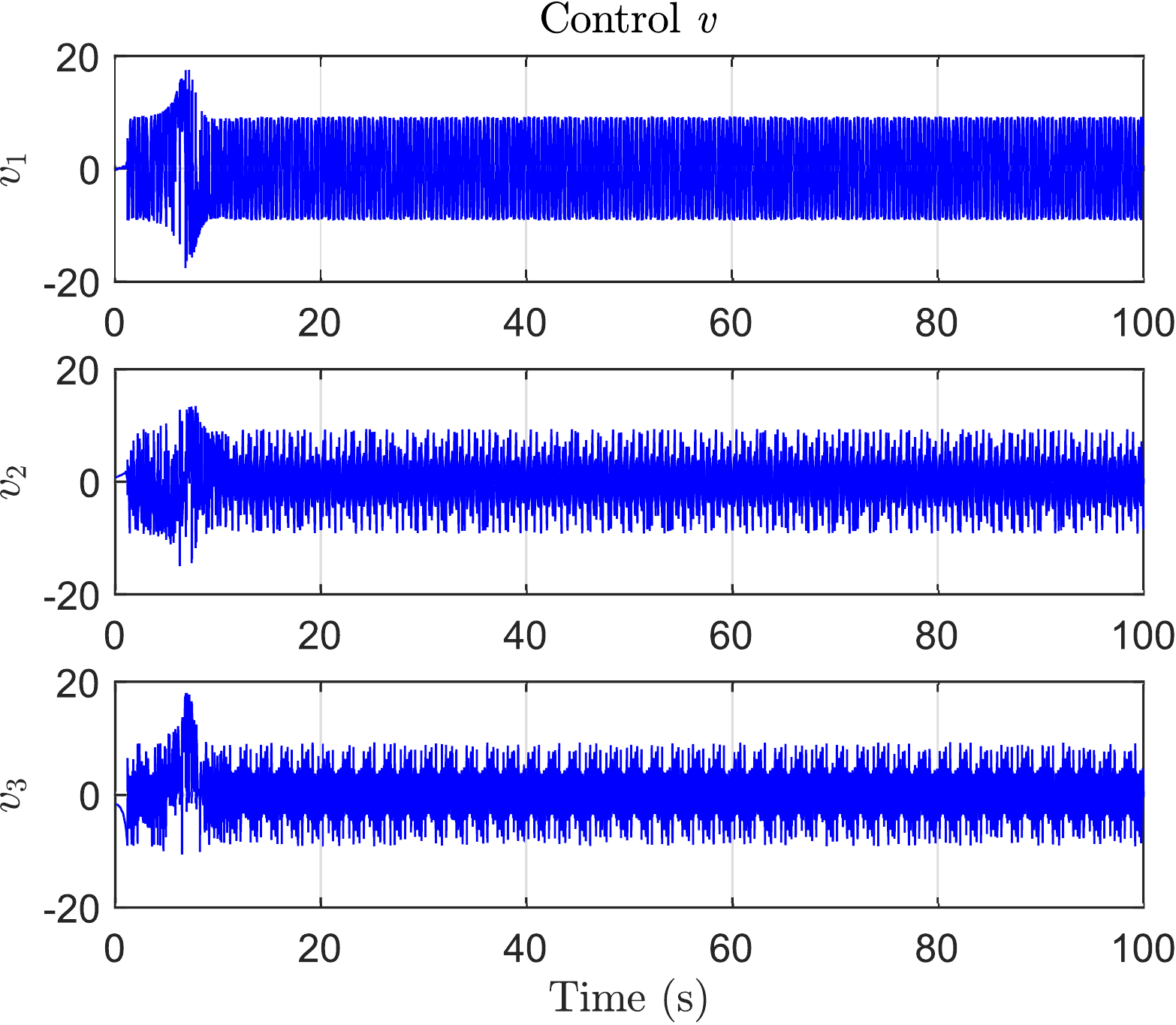}
\caption{Control torques $u$ produced by the control law~\eqref{eq:uSO3}.}
\label{fig:u}
\end{figure}

\begin{figure}
\centering
\includegraphics[width=0.85\columnwidth]{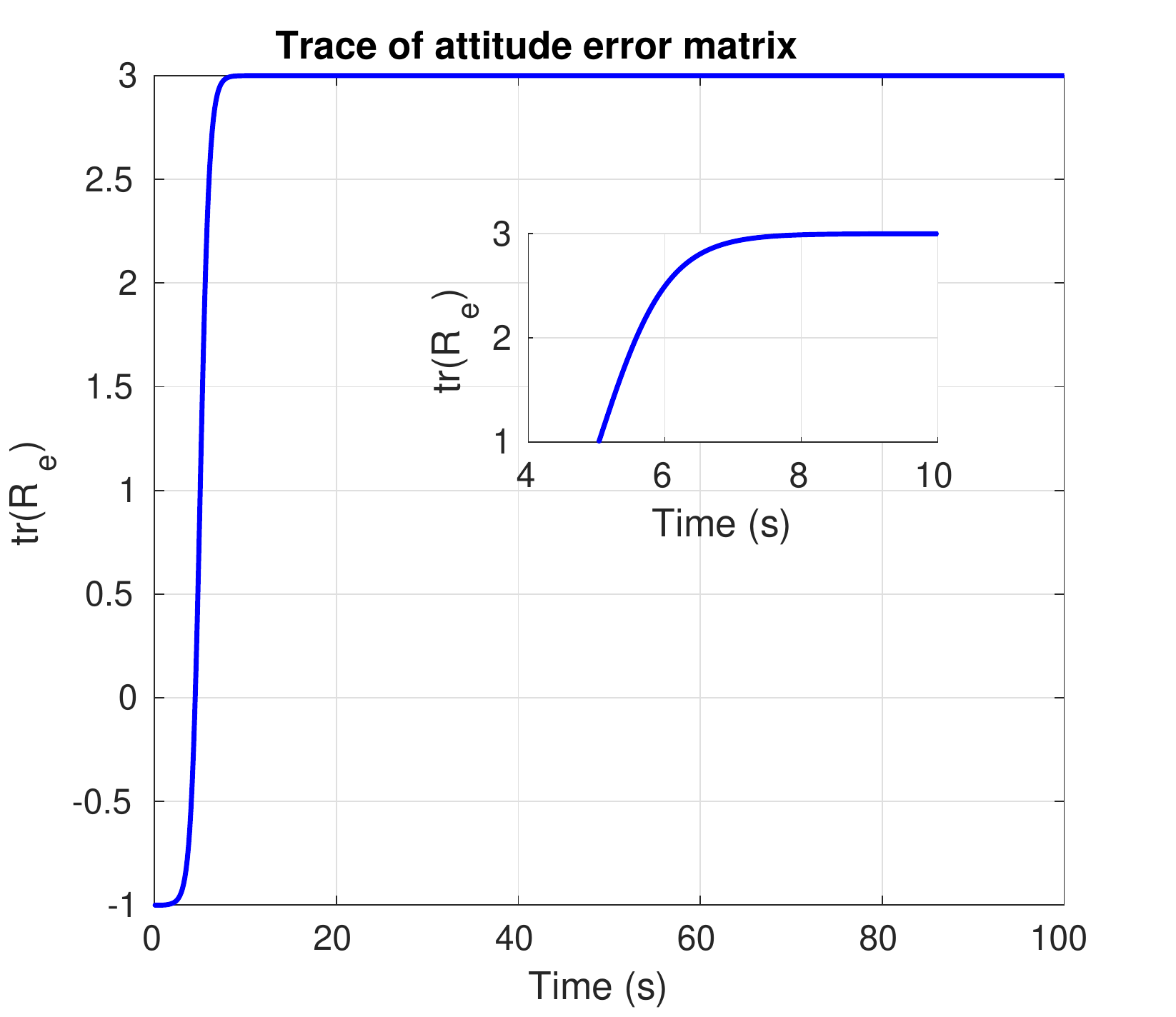}
\caption{Trace of attitude error matrix $R_e$ which shows that the closed-loop
 system~\eqref{eq:error},~\eqref{eq:uSO3} attains the required attitude, i.e.
 $R_e = R_d^T R \to I$.}
\label{fig:trace}
\end{figure}


\section{CONCLUSIONS AND FUTURE WORK}

\addtolength{\textheight}{-10.8cm}   


If the control objective is simply to specify the local behavior of a system evolving on
a manifold, then the choice of parametrization of the manifold is unimportant. If,
however, one is interested in specifying the global behavior and the system in question
evolves on a non-Euclidean manifold, then such choice becomes crucial. In the
specific attitude control problem, quaternions supersede Euler angles in that
there are no singularities, but still ail from the non-uniqueness problem. Rotational
matrices, on the other hand, are singularity free and unique.      

We have proposed a sliding surface defined in terms of rotational matrices. The resulting
controller is thus singularity free and does not induce the unwinding phenomenon. It
is shown that the reduced order system is almost globally asymptotically stable,
and that the sliding surface is in fact a Lie subgroup of the state space $SO(3)\times\RE^3$.
The latter property ensures the smoothness of the sliding surface and can be used as a
general design principle for systems that live on other Lie groups.  

\bibliographystyle{IEEEtran}
\bibliography{so3}      

\end{document}